\documentclass[12pt]{amsart}
\usepackage{geometry}                
\usepackage{graphicx}
\usepackage{amssymb}
\usepackage{setspace}
\usepackage{mathrsfs}
\usepackage{epstopdf}
\usepackage{fancyhdr}
\usepackage{amsmath}
\usepackage{yfonts}
\usepackage{amsthm}
\usepackage{eucal}
\usepackage{amssymb}
\usepackage{enumerate}

\newtheorem{theorem}{Theorem}[section]
\newtheorem{lemma}[theorem]{Lemma}

\begin{document}

\title{Upper bounds for the moments of $\zeta'(\rho)$}
\author{Micah B. Milinovich}        
\address{Department of Mathematics \\ University of Rochester \\ Rochester, NY  14627  USA}
\email{micah@math.rochester.edu}

\begin{abstract}
Assuming the Riemann Hypothesis, we obtain an upper bound for the $2k$th moment of the derivative of the Riemann zeta-function averaged over the non-trivial zeros of $\zeta(s)$ for every positive integer $k$. Our bounds are nearly as sharp as the conjectured asymptotic formulae for these moments.  
\end{abstract}

\maketitle
\begin{onehalfspacing}
\section{Introduction \& statement of the main results}
Let $\zeta(s)$ denote the Riemann zeta-function.  This article is concerned with estimating discrete moments of the form
\begin{equation}\label{Jk}
J_{k}(T) =\frac{1}{N(T)} \sum_{0<\gamma\leq T} \big|\zeta'(\rho)\big|^{2k}
\end{equation}
where $k\in\mathbb{N}$ and the sum runs over the non-trivial (complex) zeros $\rho=\beta+i\gamma$ of $\zeta(s)$.  As usual, the function 
\begin{equation}\label{NT}
N(T) = \sum_{0<\gamma\leq T} 1 = \frac{T}{2\pi}\log\frac{T}{2\pi}-\frac{T}{2\pi} + O(\log T)
\end{equation}
denotes the number of zeros of $\zeta(s)$ up to a height $T$ counted with multiplicity.  

It is an open problem to determine the behavior of $J_{k}(T)$ as $k$ varies. Independently, Gonek \cite{Gon89} and Hejhal \cite{Hej89} have conjectured that
\begin{equation}\label{GH}
J_{k}(T) \asymp (\log T)^{k(k+2)}
\end{equation}
for fixed $k\in\mathbb{R}$ as $T\rightarrow\infty$.  Though widely believed for positive values of $k$, there is evidence to suggest that this conjecture is false for $k\leq -3/2$.  

Until recently, estimates in agreement with (\ref{GH}) were only known in a few cases.  Assuming the Riemann Hypothesis (which asserts that $\beta=\tfrac{1}{2}$ for each non-trivial zero of $\zeta(s)$), Gonek \cite{Gon84} has shown that $J_{1}(T)\sim \frac{1}{12}(\log T)^{3}$ and Ng \cite{Ng04} has proved that $J_{2}(T)\asymp (\log T)^{8}$.  Confirming a conjecture of Conrey and Snaith (section 7.1 of \cite{CS}), the author \cite{Mil07a} has calculated the lower-order terms in the asymptotic expression for $J_{1}(T)$.  Under the additional assumption that the zeros of $\zeta(s)$ are simple, Gonek  \cite{Gon89} has shown that $J_{-1}(T) \gg (\log T)^{-1}$ and conjectured \cite{Gon99} that $J_{-1}(T) \sim \frac{6}{\pi^{2}}(\log T)^{-1}$.  In addition, there are a few related unconditional results where the sum in (\ref{Jk}) is restricted to the simple zeros of $\zeta(s)$ with $\beta=\tfrac{1}{2}$.  See, for instance, \cite{Gar03,GS05,LS04,SS04}.

By using a random matrix model to study the behavior of the Riemann zeta-function and its derivative on the critical line, Hughes, Keating, and O'Connell \cite{HKO00} have refined Gonek's and Hejhal's conjecture in (\ref{GH}).  In particular, they conjectured a precise constant $\mathcal{D}_{k}$ such that $J_{k}(T) \sim \mathcal{D}_{k}(\log T)^{k(k+2)}$ as $T\rightarrow\infty$ for fixed $k\in\mathbb{C}$ with $\Re k>-3/2$.  Their conjecture is consistent with the results mentioned above. 

Very little is known about the moments $J_k(T)$ when $k\!>\!2$.  However, assuming the Riemann Hypothesis, one may deduce from well-known results of Littlewood (Theorems $14.14$ A-B of Titchmarsh \cite{Tit86}) that for $\sigma\geq 1/2$ and $t\geq 10$, the estimate
\begin{equation*}
 \zeta'(\sigma\!+\!it) \ll \exp\Big(\frac{C\log t}{\log\log t}\Big)
\end{equation*}
holds for some constant $C\!>\!0$.  It immediately follows that
$$ J_k(T) \ll \exp\Big(\frac{2kC\log T}{\log\log T}\Big)$$
for any $k\!\geq\! 0$.  The goal of this paper is to improve this estimate by obtaining a conditional  upper bound for $J_k(T)$ (when $k\in\mathbb{N})$ very near the conjectured order of magnitude.  In particular, we prove the following result.
\begin{theorem}\label{th1}
Assume the Riemann Hypothesis. Let $k\in \mathbb{N}$ and $\varepsilon>0$ be arbitrary.  Then for sufficiently large $T$ we have
\begin{equation*}
\frac{1}{N(T)} \sum_{0<\gamma\leq T} \big|\zeta'(\rho)\big|^{2k} \ll (\log T)^{k(k+2)+\varepsilon},
\end{equation*}
where the implied constant depends on $k$ and $\varepsilon.$
\end{theorem}

Under the assumption of the Riemann Hypothesis, N. Ng and the author \cite{MN07} have  shown that $J_{k}(T)\gg (\log T)^{k(k+2)}$ for each fixed $k\in\mathbb{N}$.  Combining this result with Theorem \ref{th1} lends strong support for the conjecture of Gonek and Hejhal concerning the behavior of $J_{k}(T)$ in the case when $k$ is a positive integer. 

Our proof of Theorem \ref{th1} is based upon a recent method of Soundararajan \cite{Sou07} that provides upper bounds for the frequency of large values of $|\zeta(\tfrac{1}{2}\!+\!it)|$.  His method relies on obtaining an inequality for $\log |\zeta(\tfrac{1}{2}+it)|$ involving a  ``short'' Dirichlet polynomial which is a smoothed approximation to the Dirichlet series for $\log\zeta(s)$.  Using mean-value estimates for high powers of this Dirichlet polynomial, he deduces upper bounds for the measure of the set $\{t\in [0,T] : \log|\zeta(\tfrac{1}{2}\!+\!it)|\geq V\}$ and from this is able to conclude that, for arbitrary positive values of $k$ and $\varepsilon$,
\begin{equation}\label{ikt}
 \frac{1}{T} \int_{0}^{T} \big|\zeta(\tfrac{1}{2}\!+\!it)\big|^{2k} \ll_{k,\varepsilon}  (\log T)^{k^{2}+\varepsilon}
 \end{equation}
Soundararajan's techniques build upon the work of Selberg \cite{Sel44,Sel46,Sel91} who studied the distribution of values of $\log\zeta(\tfrac{1}{2}\!+\!it)$ in the complex plane.

Since $\log \zeta'(s)$ does not have a Dirichlet series representation, it is not clear that $\log|\zeta'(\tfrac{1}{2}\! +\! it)|$ can be approximated by a Dirichlet polynomial.\footnote{Hejhal \cite{Hej89} studied the distribution of $\log|\zeta'(\tfrac{1}{2}\!+\!it)|$ by a method that does not directly involve the use of Dirichlet polynomials.}  
For this reason, we do not study the distribution of the values of $\zeta'(\rho)$ directly, but instead examine the frequency of large values of $|\zeta(\rho\!+\! \alpha)|$, where $\alpha\in\mathbb{C}$ is a small shift away from a zero $\rho$ of $\zeta(s)$.  This requires deriving an inequality for $\log |\zeta(\sigma\!+\!it)|$ involving a short Dirchlet polynomial that holds uniformly for values of  $\sigma$ in a small interval to the right of, and including, $\sigma\!=\!\tfrac{1}{2}$.  Using a result of Gonek (Lemma \ref{gonek} below), we estimate high powers of this Dirichlet polynomial averaged over the zeros of the zeta-function and are able to derive upper bounds for the frequency of large values of $|\zeta(\rho\!+\!\alpha)|$.  Using  this information we prove the following theorem.  
\begin{theorem}\label{th2}
Assume the Riemann Hypothesis.  Let $\alpha\in\mathbb{C}$ with $|\alpha|\leq1$ and $|\Re\alpha -\tfrac{1}{2}|\leq (\log T)^{-1}$.  Let $k\in \mathbb{R}$ with $k>0$ and let $\varepsilon>0$ be arbitrary.  Then for sufficiently large $T$ the inequality
\begin{equation*}
\frac{1}{N(T)} \sum_{0<\gamma\leq T}\big|\zeta(\rho\!+\!\alpha)\big|^{2k} \ll_{k,\varepsilon} (\log T)^{k^{2}+\varepsilon}
\end{equation*}
holds uniformly in $\alpha$.
\end{theorem}

Comparing the result of Theorem \ref{th2} with the estimate in (\ref{ikt}), we see that our theorem provides essentially the same upper bound (up to the implied constant) for discrete averages of the Riemann zeta-function near its zeros as can be obtained for continuous moments of $|\zeta(\tfrac{1}{2}\!+\!it)|$ by using the methods in \cite{Sou07}.  There has been some previous work on discrete mean-value estimates of the zeta-function that are of a form that is similar to the sum appearing in Theorem \ref{th2}.  For instance, see the results of Gonek \cite{Gon84}, Fujii \cite{Fuj95}, and Hughes \cite{Hug03}.  

We deduce Theorem \ref{th1} from Theorem \ref{th2} since, by Cauchy's integral formula, we can use bounds for $\zeta(s)$ near its zeros to recover bounds on the values for $\zeta'(\rho)$.  For a precise statement of this idea, see Lemma \ref{CE} below.  Our proof allows us only to establish Theorem \ref{th1} when $k$ is a positive integer despite the fact that Theorem \ref{th2} holds for all  $k>0$.

The author would like to thank Steve Gonek for his support and encouragement and also  Soundararajan for a helpful conversation.


\section{An inequality for $\log|\zeta(\sigma\!+\!it)|$ when $\sigma\geq\tfrac{1}{2}$.}

Throughout the remainder of this article, we use $s\!=\!\sigma\!+\!it$ to denote a complex variable and use $p$ to denote a prime number.  We let $\lambda_{0}=.5671...$ be the unique positive real number satisfying $e^{-\lambda_{0}} = \lambda_{0}$.  Also, we put $\sigma_{\lambda}= \sigma_{\lambda,x}=\tfrac{1}{2}+\tfrac{\lambda}{\log x}$ and let
$$
\log^{+}|x| = \left\{ \begin{array}{ll}
 0, &\mbox{ if $|x|<1$,} \\
\log|x|, &\mbox{ if $|x|\geq 1$.}
       \end{array} \right.
$$
As usual, we denote by $\Lambda(\cdot)$ the arithmetic function defined by $\Lambda(n)=\log p$ when $n=p^{k}$ and $\Lambda(n)=0$ when $n\neq p^{k}$.  The main result of this section is the following lemma.

\begin{lemma}\label{logzineq}
Assume the Riemann Hypothesis.  Let $\tau = |t|+3$ and $2\leq x\leq \tau^{2}$.  Then, for any $\lambda$ with $\lambda_{0}\leq\lambda\leq \tfrac{\log x}{4}$, the estimate
\begin{equation}\label{logzineq2}
\log^{+}\big|\zeta(\sigma\!+\!it)\big| \leq \Bigg| \sum_{n\leq x} \frac{\Lambda(n)}{n^{\sigma_\lambda+it}\log n}\frac{\log x/n}{\log x}\Bigg| + \frac{(1\!+\!\lambda)}{2}\frac{\log \tau}{\log x} + O(1)
\end{equation}
holds uniformly for $1/2\leq\sigma\leq\sigma_{\lambda}$.
\end{lemma}

In \cite{Sou07}, Soundararajan proved an inequality similar to Lemma \ref{logzineq} for the function $\log\big|\zeta(\tfrac{1}{2}\!+\!it)\big|$.  In his case, when $\zeta(\tfrac{1}{2}\!+\!it)\neq 0$, an inequality slightly stronger than (\ref{logzineq2}) holds with the constant $\lambda_{0}$ replaced by $\delta_{0}=.4912...$ where $\delta_{0}$ is the unique positive real number satisfying $e^{-\delta_{0}} = \delta_{0}+\tfrac{1}{2}\delta_{0}^{2}$.  Our proof of the above lemma is a modification of his argument. \\

\noindent{\it Proof of Lemma \ref{logzineq}.} We assume that $|\zeta(\sigma\!+\!it)|\!\geq\!1$, as otherwise the lemma holds for a trivial reason.  In particular, we are assuming that $\zeta(\sigma\!+\!it)\!\neq\!0$. Assuming the Riemann Hypothesis, we denote a non-trivial zeros of $\zeta(s)$ as $\rho=\tfrac{1}{2}+i\gamma$ and define the function
$$ F(s) = \Re \ \sum_{\rho} \frac{1}{s\!-\!\rho} = \sum_{\rho}\frac{\sigma\!-\!\tfrac{1}{2}}{(\sigma\!-\!\tfrac{1}{2})^{2}+(t\!-\!\gamma)^{2}}.$$
Notice that $F(s)\geq 0$ whenever $\sigma\geq\tfrac{1}{2}$ and $s\neq\rho$.  The partial fraction decomposition of $\zeta'(s)/\zeta(s)$ (equation (2.12.7) of Titchmarsh \cite{Tit86}) says that for $s\neq 1$ and $s$ not coinciding with a zero of $\zeta(s)$, we have
\begin{equation}\label{part}
 \frac{\zeta'}{\zeta}(s) = \sum_{\rho} \Big(\frac{1}{s\!-\!\rho}+\frac{1}{\rho}\Big) -\frac{1}{2}\frac{\Gamma'}{\Gamma}\big(\tfrac{1}{2}s\!+\!1\big) -\frac{1}{s\!-\!1} +B
 \end{equation}
where the constant $B= \log 2\pi\!-\!1\!-\!2 \gamma_{0}$; $\gamma_{0}$ denotes Euler's constant.  Taking the real part of each term in (\ref{part}), we find that 
\begin{equation}\label{part2}
 - \Re \ \frac{\zeta'}{\zeta}(s) = -\Re \ \frac{1}{2}\frac{\Gamma'}{\Gamma}\big(\tfrac{1}{2}s\!+\!1\big)  - F(s) +O(1) .
 \end{equation}
Stirling's asymptotic formula for the gamma function implies that 
\begin{equation} \label{stirform}
 \frac{\Gamma'}{\Gamma}(s) = \log s -\frac{1}{2s}+O\big(|s|^{-2}\big)
 \end{equation}
for $\delta>0$ fixed, $|\arg s| < \pi -\delta$, and $|s|>\delta$ (see Appendix A.7 of Ivi\`{c} \cite{Ivic85}).  By combining (\ref{part2}) and (\ref{stirform}) with the observation that $F(s)\geq 0$, we find that  
\begin{equation}\label{new6}
\begin{split}
- \Re \ \frac{\zeta'}{\zeta}(s) &= \tfrac{1}{2}\log\tau - F(s) +O(1) 
\\
& \leq \tfrac{1}{2}\log \tau  +O(1).
\end{split}
\end{equation}
uniformly for $\tfrac{1}{2}\leq \sigma\leq 1$. Consequently, the inequality
\begin{equation}\label{new7}
\begin{split}
\log|\zeta(\sigma\!+\!it)|-\log|\zeta(\sigma_{\lambda}\!+\!it)| &= \Re\int_{\sigma}^{\sigma_{\lambda}} \Big[-\frac{\zeta'}{\zeta}(u\!+\!it)\Big] du
\\
&\leq \big(\sigma_{\lambda}\!-\!\sigma\big)\big(\tfrac{1}{2}\log \tau  +O(1)\big)
\\
&\leq \big(\sigma_{\lambda}\!-\!\tfrac{1}{2}\big)\big(\tfrac{1}{2}\log \tau  +O(1)\big)
\end{split}
\end{equation}
holds uniformly for $\tfrac{1}{2}\leq\sigma\leq\sigma_{\lambda}.$ 

To complete the proof of the lemma, we require an upper bound for $\log|\zeta(\sigma_{\lambda}\!+\!it)|$ which, in turn, requires an additional identity for $\zeta'(s)/\zeta(s)$.  Specifically, for $s\neq 1$ and $s$ not coinciding with a zero of $\zeta(s)$, we have
\begin{equation}
\begin{split}
-\frac{\zeta'}{\zeta}(s) = \sum_{n\leq x} \frac{\Lambda(n)}{n^{s}}\frac{\log(x/n)}{\log x} &+ \frac{1}{\log x} \Big(\frac{\zeta'}{\zeta}(s)\Big)'+ \frac{1}{\log x}\sum_{\rho}\frac{x^{\rho-s}}{(\rho\!-\!s)^{2}}
\\
& - \frac{1}{\log x}\frac{x^{1-s}}{(1\!-\!s)^{2}} +\frac{1}{\log x} \sum_{k=1}^{\infty} \frac{x^{-2k-s}}{(2k\!+\!s)^{2}}.
\end{split}
\end{equation}
This identity is due to Soundararajan (Lemma 1 of \cite{Sou07}). Integrating over $\sigma$ from $\sigma_{\lambda}$ to $ \infty$, we deduce from the above identity that
\begin{equation}\label{new8}
\begin{split}
\log|\zeta(\sigma_{\lambda}\!+\!it)| = \Re \ \sum_{n\leq x} &\frac{\Lambda(n)}{n^{\sigma_\lambda+it}\log n}\frac{\log x/n}{\log x} -\frac{1}{\log x} \Re \ \frac{\zeta'}{\zeta}(\sigma_{\lambda}\!+\!it)
\\
&+ \frac{1}{\log x} \sum_{\rho} \Re \int_{\sigma_{\lambda}}^{\infty} \frac{x^{\rho-s}}{(\rho\!-\!s)^{2}} d\sigma +O\Big(\frac{1}{\log x}\Big).
\end{split}
\end{equation}
We now estimate the second and third terms on the right-hand side of this expression.  Arguing as above, using (\ref{part}) and (\ref{stirform}), we find that
\begin{equation}\label{new10}
\Re \ \frac{\zeta'}{\zeta}(\sigma_{\lambda}\!+\!it) = \tfrac{1}{2}\log \tau - F(\sigma_{\lambda}\!+\!it) +O(1).
\end{equation}
Also, observing that
\begin{equation}\label{new9}
\begin{split}
\sum_{\rho}\Big| \int_{\sigma_{\lambda}}^{\infty} \frac{x^{\rho-s}}{(\rho\!-\!s)^{2}}d\sigma\Big|&\leq \sum_{\rho} \int_{\sigma_{\lambda}}^{\infty} \frac{x^{1/2-\sigma}}{|\rho\!-\!s|^{2}}d\sigma
\\
& = \sum_{\rho} \frac{x^{1/2-\sigma_{\lambda}}}{|\rho\!-\!\sigma_{\lambda}\!-\!it|^{2}\log x}= \frac{x^{1/2-\sigma_{\lambda}}F(\sigma_{\lambda}\!+\!it)}{(\sigma_{\lambda}\!+\!it)\log x},
\end{split}
\end{equation}
and combining (\ref{new10}) and (\ref{new9}) with (\ref{new8}), we see that
\begin{equation*}
\begin{split}
\log|\zeta(\sigma_{\lambda}\!+\!it)| \leq \Re &\sum_{n\leq x} \frac{\Lambda(n)}{n^{\sigma_\lambda+it}\log n}\frac{\log x/n}{\log x} + \frac{1}{2}\frac{\log \tau}{\log x}
\\
& + \frac{F(\sigma_\lambda\!+\!it)}{\log x} \Big(\frac{x^{1/2-\sigma_\lambda}}{(\sigma_\lambda\!-\!\tfrac{1}{2})\log x} - 1\Big) +  O\Big(\frac{1}{\log x}\Big).
\end{split}
\end{equation*}
If $\lambda \geq \lambda_{0}$, then the term on the right-hand side involving $F(\sigma_\lambda\!+\!it)$ is less than or equal to zero, so omitting it does not change the inequality.  Thus,
\begin{equation}\label{new12}
\begin{split}
\log|\zeta(\sigma_{\lambda}\!+\!it)| \leq \Re &\sum_{n\leq x} \frac{\Lambda(n)}{n^{\sigma_\lambda+it}\log n}\frac{\log x/n}{\log x} + \frac{1}{2}\frac{\log \tau}{\log x} + O\Big(\frac{1}{\log x}\Big).
\end{split}
\end{equation}
Since we have assumed that $|\zeta(\sigma\!+\!it)|\geq 1$, the lemma now follows by combining the inequalities in (\ref{new7}) and (\ref{new12}) and then taking absolute values.


\section{A variation of lemma \ref{logzineq}}
In this section, we prove a version of Lemma \ref{logzineq} in which the sum on the right-hand side of the inequality is restricted just to the primes.  A sketch of the proof of the lemma appearing below has been given previously by Soundararajan (see \cite{Sou07}, Lemma 2).  Our proof is different and the details are provided for completeness.

\begin{lemma}\label{new1}
Assume the Riemann Hypothesis.  Put $\tau = |t|+e^{30}$.  Then, for $\sigma\geq \tfrac{1}{2}$ and $2\leq x\leq \tau^{2}$, we have
\begin{equation*}
\left| \sum_{n\leq x} \frac{\Lambda(n)}{n^{\sigma+it}\log n}\frac{\log x/n}{\log x} - \sum_{p\leq x} \frac{1}{p^{\sigma+it}}\frac{\log x/n}{\log x}\right| = O\big(\log\log\log\tau\big).
\end{equation*}
As a consequence, for any $\lambda$ with $\lambda_{0}\leq\lambda\leq \tfrac{\log x}{4}$, the estimate
\begin{equation*}
\log^{+}\big|\zeta(\sigma\!+\!it)\big| \leq \Bigg| \sum_{p\leq x} \frac{1}{p^{\sigma_\lambda+it}}\frac{\log x/p}{\log x}\Bigg| + \frac{(1\!+\!\lambda)}{2}\frac{\log \tau}{\log x} + O\big(\log\log\log\tau \big)
\end{equation*}
holds uniformly for $\tfrac{1}{2}\leq\sigma\leq\sigma_{\lambda}$ and $2\leq x \leq \tau^{2}$. 
\end{lemma}
\begin{proof} First we observe that, for $\sigma\geq \tfrac{1}{2}$,
\begin{equation*}
\begin{split}
 \sum_{n\leq x} \frac{\Lambda(n)}{n^{s}\log n}\frac{\log x/n}{\log x} - \sum_{p\leq x} \frac{1}{p^{s}}\frac{\log x/p}{\log x} &= \frac{1}{2}\sum_{p \leq \sqrt{x}} \frac{1}{ p^{2s}}\frac{\log \sqrt{x}/n}{\log \sqrt{x}} + O(1).
\\
&= \frac{1}{2} \sum_{n\leq \sqrt{x}} \frac{\Lambda(n)}{n^{2s}\log n}\frac{\log \sqrt{x}/n}{\log \sqrt{x}} + O(1).
\end{split}
\end{equation*}
Thus, if we let $w=u+iv$ and $\nu = |v|+e^{30}$, the lemma will follow if we can show that
\begin{equation}\label{new3}
\sum_{n\leq z} \frac{\Lambda(n)}{n^{w}\log n}\frac{\log z/n}{\log z} = O\big(\log\log\log\nu \big)
 \end{equation}
uniformly for $u \geq 1$ and $2\leq z\leq \nu. \ $  In what follows, we can assume that $z \geq (\log \nu)^{2}$ as otherwise 
$$ \sum_{n\leq z} \frac{\Lambda(n)}{n^{w}\log n}\frac{\log z/n}{\log z} \ll \sum_{p < \log^{2}\nu} \frac{1}{p} \ll \log \log \log \nu.$$

Let $c\!=\!\max(2,1\!+\!u)$.  Then, by expressing $\frac{\zeta'}{\zeta}(s\!+\!w)$ as a Dirichlet series and interchanging the order of summation and integration (which is justified by absolute convergence), it follows that
\begin{eqnarray*}
\frac{1}{2\pi i} \int_{c-i\infty}^{c+i\infty} \Big[-\frac{\zeta'}{\zeta}(s\!+\!w)\Big] z^{s} \frac{ds}{s^{2}} &=& \frac{1}{2\pi i} \int_{c-i\infty}^{c+i\infty} \Bigg[\sum_{n=1}^{\infty} \frac{\Lambda(n)}{n^{s+w}} \Bigg] z^{s} \frac{ds}{s^{2}}
\\
&=& \frac{1}{2\pi i}\sum_{n=1}^{\infty} \frac{\Lambda(n)}{n^{w}} \int_{c-i\infty}^{c+i\infty} \Big(\frac{z}{n}\Big)^{s} \frac{ds}{s^{2}}
\\
&=& \sum_{n\leq z} \frac{\Lambda(n)}{n^{w}} \log(z/n).
\end{eqnarray*}
Here we have made use of the standard identity
\begin{equation*}
\frac{1}{2\pi i} \int_{c-i\infty}^{c+i\infty} x^{s} \frac{ds}{s^{2}} = \left\{ \begin{array}{ll}
 \log x, &\mbox{ if $x\geq 1$,} \\
  0, &\mbox{ if $0\leq x<1$,}
       \end{array} \right.
 \end{equation*}
which is valid for $c\!>\!0$.
By moving the line of integration in the integral left to $\Re s= \sigma = \tfrac{3}{4}-u$, we find by the calculus of residues that
\begin{equation}\label{residont}
\begin{split}
 \sum_{n\leq z} \frac{\Lambda(n)}{n^{w}} \log(z/n) =  -(\log z) & \frac{\zeta'}{\zeta}(w) -  \Big(\frac{\zeta'}{\zeta}(w)\Big)' +  \frac{z^{1-w}}{(w\!-\!1)^{2}} 
 \\
 &+ \frac{1}{2\pi i} \int_{ \tfrac{3}{4}-u-i\infty}^{ \tfrac{3}{4}-u+i\infty} \Big[-\frac{\zeta'}{\zeta}(s\!+\!w)\Big] z^{s} \frac{ds}{s^{2}}.
 \end{split}
 \end{equation}
That there are no residues obtained from poles of the integrand at the non-trivial zeros of $\zeta(s)$ follows from the Riemann Hypothesis.  To estimate the integral on the right-hand side of the above expression, we use Theorem $14.5$ of Titchmarsh \cite{Tit86}, namely, that if the Riemann Hypothesis is true, then
\begin{equation}\label{hgh}
 \Big|\frac{\zeta'}{\zeta}(\sigma\!+\!it)\Big| \ll (\log \tau)^{2-2\sigma}
 \end{equation}
uniformly for $\tfrac{5}{8} \leq \sigma \leq \tfrac{7}{8}, $ say.  Using (\ref{hgh}), it immediately follows that
$$ \int_{ \tfrac{3}{4}-u-i\infty}^{ \tfrac{3}{4}-u+i\infty} \Big[-\frac{\zeta'}{\zeta}(s\!+\!w)\Big] z^{s} \frac{ds}{s^{2}} \ll z^{3/4-u}\sqrt{\log\nu}. $$
Inserting this estimate into equation (\ref{residont}) and dividing by $\log z$, it follows that
\begin{equation}\label{residonot}
\begin{split}
\sum_{n\leq z} \frac{\Lambda(n)}{n^{w}} \frac{\log(z/n)}{\log z} &=  - \frac{\zeta'}{\zeta}(w) -  \frac{1}{\log z} \Big(\frac{\zeta'}{\zeta}(w)\Big)' 
\\
&\quad \quad +  \frac{z^{1-w}}{(w\!-\!1)^{2} \log z} + O\Big(\frac{z^{3/4-u}}{\log z}\sqrt{\log\nu}\Big) .
\end{split}
\end{equation}
Integrating the expression in (\ref{residonot}) from $\infty$ to $u$ (along the line $\sigma+i\nu$, $u\leq \sigma <\infty)$, we find that
\begin{equation*}
\begin{split}
\sum_{n\leq z} \frac{\Lambda(n)}{n^{w}\log n} \frac{\log(z/n)}{\log z} &=  \log \zeta(w)  +  \frac{1}{\log z} \frac{\zeta'}{\zeta}(w)
\\
& \quad +  O\Big( \frac{z^{1-u}}{\nu^{2} (\log z)^{2}} +\frac{z^{3/4-u}}{(\log z)^{2}}\sqrt{\log\nu}\Big) .
\end{split}
\end{equation*}
Assuming the Riemann Hypothesis, we can estimate the terms on the right-hand side of of the above expression by invoking the bounds 
\begin{equation}\label{HBbounds}
 |\log \zeta(\sigma\!+\!it)| \ll \log\log\log \tau \quad \quad \text{ and } \quad \quad \Big|\frac{\zeta'}{\zeta}(\sigma\!+\!it)\Big| \ll \log \log \tau 
 \end{equation}
which hold uniformly for $\sigma\!\geq\!1$ and $|t|\!\geq\!1$.  (For a discussion of such estimates see Heath-Brown's notes following Chapter $14$ in Titchmarsh \cite{Tit86}.)  Using the estimates in (\ref{HBbounds}) and recalling that we are assuming that $u\geq 1$ and $z\geq (\log\nu)^{2}$, we find that
\begin{equation*}
\begin{split}
\sum_{n\leq z} \frac{\Lambda(n)}{n^{w}\log n} \frac{\log(z/n)}{\log z} &\ll  \log\log\log \nu +\frac{\log\log\nu}{\log z} + \frac{z^{1-u}}{\nu^{2} (\log z)^{2}} + z^{-1/4}\frac{\sqrt{\log\nu}}{(\log z)^{2}}
\\
&\ll  \log\log\log \nu .
\end{split}
\end{equation*}
This establishes (\ref{new3}) and, thus, the lemma.
\end{proof}


\section{A sum over the zeros of $\zeta(s)$}

In this section we prove an estimate for the mean-square of a Dirichlet polynomial averaged over the zeros of $\zeta(s)$.  Our estimate follows from the Landau-Gonek explicit formula.

\begin{lemma}\label{gonek}
Let $x,T>1$ and let $\rho=\beta+i\gamma$ denote a non-trivial zero of $\zeta(s)$.  Then
\begin{equation*}
\begin{split}
\sum_{0<\gamma \leq T} x^{\rho} &= -\frac{T}{2\pi}\Lambda(x) + O\big(x\log(2xT)\log\log(3x)\big)
\\
&\quad + O\Big(\log x \min\Big(T,\frac{x}{\langle x\rangle}\Big)\Big) + O\Big(\log(2T)\min\Big(T,\frac{1}{\log x}\Big)\Big), 
\end{split}
\end{equation*}
where $\langle x\rangle$ denotes the distance from $x$ to the nearest prime power other than $x$ itself, $\Lambda(x)=\log p$ if $x$ is a positive integral power of a prime $p$, and $\Lambda(x)=0$ otherwise.
\end{lemma}
\begin{proof}
This is due to Gonek \cite{Gon85,Gon93}. 
\end{proof}

\begin{lemma}\label{mvt}
Assume the Riemann Hypothesis and let $\rho=\tfrac{1}{2}+i\gamma$ denote a non-trivial zero of $\zeta(s)$.  For any sequence of complex numbers $\mathscr{A}=\{a_{n}\}_{n=1}^{\infty}$  define, for $\xi\geq1$,  $$m_{\xi}=m_{\xi}(\mathscr{A}) = \max_{1\leq n\leq \xi}\big(1,|a_{n}|\big).$$ Then for $3\leq\xi\leq T(\log T)^{-1}$ and any complex number $\alpha$ with $\Re\alpha\geq 0$ we have
\begin{equation}\label{mvt1}
\sum_{0<\gamma\leq T} \Bigg|\sum_{n\leq\xi}\frac{a_{n}}{n^{\rho+\alpha}}\Bigg|^{2} \ll  m_{\xi} T\log T\sum_{n\leq\xi}\frac{|a_{n}|}{n},
\end{equation}
where the implied constant is absolute (and independent of $\alpha$).
\end{lemma}

\begin{proof}  Assuming the Riemann Hypothesis, we note that $1-\rho=\bar{\rho}$ for any non-trivial zero $\rho=\tfrac{1}{2}+i\gamma$ of $\zeta(s)$.  This implies that
$$ \Bigg|\sum_{n\leq\xi}\frac{a_{n}}{n^{\rho+\alpha}}\Bigg|^{2} = \sum_{m\leq \xi}  \sum_{n\leq \xi} \frac{a_{m}}{m^{\rho+\alpha}} \frac{\overline{a_n}}{n^{1-\rho+\bar{\alpha}}}, $$
and, moreover, that
\begin{equation*}\label{unfold}
\begin{split}
\sum_{0<\gamma\leq T} \left|\sum_{n\leq\xi}\frac{a_{n}}{n^{\rho+\alpha}}\right|^{2} =&N(T)\sum_{n\leq\xi}\frac{|a_{n}|^{2}}{n^{1+2\Re\alpha}} + 2\Re\sum_{m\leq \xi}\frac{a_{m}}{m^{\alpha}}\sum_{m<n\leq\xi}\frac{\overline{a_{n}}}{n^{1+\bar{\alpha}} }\sum_{0<\gamma\leq T} \Big(\frac{n}{m}\Big)^{\rho}
\end{split}
\end{equation*}
where  $N(T) \sim \tfrac{T}{2\pi} \log T$ denotes the number of zeros $\rho$ with $0<\gamma\leq T$.  Since $\Re\alpha\geq 0$, it follows that
\begin{equation*}
N(T) \sum_{n\leq\xi}\frac{|a_{n}|^{2}}{n^{1+2\Re\alpha}} \ll T\log T\sum_{n\leq\xi}\frac{|a_{n}|^{2}}{n} \ll m_{\xi} T\log T  \sum_{n\leq\xi}\frac{|a_{n}|}{n}.
\end{equation*}
Appealing to Lemma \ref{gonek}, we find that
\begin{equation*}
\begin{split}
\sum_{m\leq \xi}\frac{a_{m}}{m^{\alpha}}\sum_{n<m}\frac{\overline{a_{n}}}{n^{1+\bar{\alpha}} }\sum_{0<\gamma\leq T} \Big(\frac{n}{m}\Big)^{\rho} &= \Sigma_{1}+\Sigma_{2}+\Sigma_{3}+\Sigma_{4},
\end{split}
\end{equation*}
where
\begin{equation*}
\begin{split}
\Sigma_{1}&= -\frac{T}{2\pi}\sum_{m\leq\xi}\frac{a_{m}}{m^{\alpha}}\sum_{m<n\leq\xi}\frac{\overline{a_{n}}}{n^{1+\bar{\alpha}}}\Lambda\Big(\frac{n}{m}\Big),
\\
\Sigma_{2}&= O\left(\log T\log\log T \sum_{m\leq \xi}\frac{|a_{m}|}{m^{1+\Re\alpha}}\sum_{m<n\leq\xi}\frac{|a_{n}|}{n^{\Re\alpha}} \right),
\\
\Sigma_{3}&= O\left(\sum_{m\leq \xi}\frac{|a_{m}|}{m^{1+\Re\alpha}}\sum_{m<n\leq\xi}\frac{|a_{n}|}{n^{\Re\alpha} }  \frac{\log\frac{m}{n}}{\langle \frac{m}{n}\rangle}\right),
\end{split}
\end{equation*}
and
\begin{equation*}
\begin{split}
\Sigma_{4}&= O\left(\log T \sum_{m\leq \xi}\frac{|a_{m}|}{m^{\Re\alpha}}\sum_{m<n\leq\xi}\frac{|a_{n}|}{n^{1+\Re\alpha} \log\frac{n}{m}}\right). \quad \
\end{split}
\end{equation*}
We estimate $\Sigma_{1}$ first.  Making the substitution $n=mk$, we re-write our expression for $\Sigma_{1}$ as
$$-\frac{T}{2\pi}\sum_{m\leq\xi}\frac{a_{m}}{m^{\alpha}}\sum_{k\leq\frac{\xi}{m}}\frac{\overline{a_{mk}}\cdot\Lambda(k)}{(mk)^{1+\bar{\alpha}}}=-\frac{T}{2\pi}\sum_{m\leq\xi}\frac{a_{m}}{m^{1+2\Re\alpha}}\sum_{k\leq\frac{\xi}{m}}\frac{\overline{a_{mk}}\cdot\Lambda(k)}{k^{1+\bar{\alpha}}}.$$
Again using the assumption that $\Re\alpha\geq 0$, we find that
\begin{equation*}
\Sigma_{1} \ll m_{\xi} T\sum_{n\leq\xi}\frac{|a_{n}|}{n}\sum_{m\leq\frac{\xi}{n}}\frac{\Lambda(m)}{m} \ll  m_{\xi} T\log T \sum_{n\leq\xi}\frac{|a_{n}|}{n}.
\end{equation*}
Here we have made use of the standard estimate $ \sum_{m\leq\xi}\frac{\Lambda(m)}{m} \ll \log \xi.$  We can replace $\Re\alpha$ by $0$ in each of the sums $\Sigma_{i}$ (for $i=2,3,$ or $4$), as doing so will only make the corresponding estimates larger.  Thus, using the assumption that $3\leq \xi\leq T/\log T$, it follows that
\begin{equation*}
\begin{split}
\Sigma_{2} \ll m_{\xi} \log T \log\log T \sum_{n\leq \xi}\frac{|a_{n}|}{n} \sum_{m<n\leq\xi} 1  \ll  m_{\xi} T\log T \sum_{n\leq\xi}\frac{|a_{n}|}{n}.
\end{split}
\end{equation*}
Next, turning to $\Sigma_{3}$, we find that
\begin{equation*}
\Sigma_{3} \ll m_{\xi} \sum_{m\leq\xi}\frac{|a_{m}|}{m} \sum_{m<n\leq \xi} \frac{\log\frac{n}{m}}{\langle \frac{n}{m}\rangle }. 
\end{equation*}
Writing $n$ as $qm+\ell$ with $-\frac{m}{2}<\ell\leq\frac{m}{2}$, we have
\begin{equation*}
\Sigma_{3} \ll m_{\xi} \sum_{m\leq\xi}\frac{|a_{m}|}{m} \sum_{q\leq\lfloor \frac{\xi}{m} \rfloor+1}\sum_{-\frac{m}{2}<\ell\leq\frac{m}{2}} \frac{\log\big(q\!+\!\frac{\ell}{m}\big)}{\langle q\!+\!\frac{\ell}{m}\rangle },
\end{equation*}
where, as usual, $\lfloor x \rfloor$ denotes the greatest integer less than or equal to $x$.  Now $\langle q+\frac{\ell}{m}\rangle  =\frac{|\ell|}{m}$ if $q$ is a prime power and $\ell\neq0$, otherwise $\langle q+\frac{\ell}{m}\rangle $ is $\geq \frac{1}{2}$.  Using the estimate $\sum_{n\leq\xi}\Lambda(n)\ll \xi$, we now find that
\begin{equation*}
\begin{split}
\Sigma_{3} &\ll m_{\xi} \sum_{m\leq\xi}\frac{|a_{m}|}{m} \sum_{q\leq\lfloor \frac{\xi}{m} \rfloor+1} \Lambda(q)\sum_{1\leq\ell\leq\frac{m}{2}} \frac{m}{\ell} 
\\
&\quad \quad \quad \quad+ m_{\xi} \sum_{m\leq\xi}\frac{|a_{m}|}{m} \sum_{q\leq\lfloor \frac{\xi}{m} \rfloor+1} \log(q\!+\!1) \sum_{1\leq\ell\leq\frac{m}{2}} 1
\\
&\ll m_{\xi} \sum_{m\leq\xi} |a_{m}|\log m \sum_{q\leq\lfloor \frac{\xi}{m} \rfloor+1} \Lambda(q)
\\
&\quad \quad \quad \quad+ m_{\xi} \sum_{m\leq\xi} |a_{m}| \sum_{q\leq\lfloor \frac{\xi}{m} \rfloor+1} \log(q\!+\!1) 
\\
&\ll m_{\xi} (\xi \log \xi) \sum_{m\leq\xi}\frac{|a_{m}|}{m}
\\
&\ll m_{\xi}T\log T \sum_{m\leq\xi}\frac{|a_{m}|}{m}.
\end{split}
\end{equation*}
It remains to consider the contribution from $\Sigma_{4}$ which is 
\begin{equation*}
\ll m_{\xi} \log T \sum_{m\leq \xi}|a_{m}| \sum_{m<n\leq\xi} \frac{1}{n\log\frac{n}{m}} \ll m_{\xi} \log T \sum_{m\leq \xi}\frac{|a_{m}|}{m} \sum_{m<n\leq\xi} \frac{1}{\log\frac{n}{m}} \ ,
\end{equation*}
since $\frac{1}{m}>\frac{1}{n}$ if $n>m$.  Writing $n=m+\ell$, we see that
$$ \sum_{m<n\leq\xi} \frac{1}{\log\frac{n}{m}} = \sum_{1\leq\ell \leq \xi-m} \frac{1}{\log\big(1\!+\!\frac{\ell}{m}\big)} \ll \sum_{1\leq\ell \leq \xi-m}  \frac{m}{\ell} \ll m\log\xi \ll \xi\log\xi.$$
Consequently, 
\begin{equation*}
\Sigma_{4} \ll m_{\xi}T\log T \sum_{m\leq\xi}\frac{|a_{m}|}{m}.
\end{equation*}
Now, by combining estimates, we obtain the lemma.  
\end{proof}


\section{The frequency of large values of $|\zeta(\rho\!+\!\alpha)|$}
Our proof of Theorem \ref{th2} requires the following lemma concerning the distribution of values of $|\zeta(\rho\!+\!\alpha)|$ where $\rho$ is a zero of $\zeta(s)$ and $\alpha\in\mathbb{C}$ is a small shift. In what follows, $\log_{3}(\cdot)$ stands for $\log\log\log(\cdot)$.  

\begin{lemma}\label{vd}
Assume the Riemann Hypothesis.  Let $T$ be large, $V\geq 3$ a real number, and $\alpha\in\mathbb{C}$  with $|\alpha|\leq 1$ and $0\leq\Re\alpha-\tfrac{1}{2}\leq (\log T)^{-1}$.  Consider the set
$$ \mathcal{S}_{\alpha}\big(T;V\big) = \big\{\gamma\in(0,T] : \log |\zeta(\rho\!+\!\alpha)| \geq V\big\}$$
where $\rho=\tfrac{1}{2}+i\gamma$ denotes a non-trivial zero of $\zeta(s)$.  Then, the following inequalities for $\# \mathcal{S}_{\alpha}\big(T;V\big)$, the cardinality of $\mathcal{S}_{\alpha}\big(T;V\big)$, hold.
\begin{enumerate}[(i)]
\item When $\sqrt{\log\log T} \leq V \leq \log\log T$, we have
$$ \# \mathcal{S}_{\alpha}\big(T;V\big) \ll N(T) \frac{V}{\sqrt{\log\log T}} \exp\left(-\tfrac{V^{2}}{\log\log T}\Big(1-\tfrac{4}{\log_{3}T}\Big)\right).$$

\item When $ \ \log\log T\leq V\leq \tfrac{1}{2}(\log\log T)\log_{3}T$, we have
$$ \# \mathcal{S}_{\alpha}\big(T;V\big) \ll N(T)  \frac{V}{\sqrt{\log\log T}} \exp\left(-\tfrac{V^{2}}{\log\log T}\Big(1-\tfrac{4V}{(\log\log T)\log_{3}T}\Big)\right).$$

\item Finally, when $ \ V>\tfrac{1}{2}(\log\log T) \log_{3}T$, we have
$$ \# \mathcal{S}_{\alpha}\big(T;V\big) \ll N(T) \exp\left(-\tfrac{V}{201}\log V\right).$$
\end{enumerate}
Here, as usual, the function $N(T)\sim \frac{T}{2\pi} \log T$ denotes the number of zeros $\rho$ of $\zeta(s)$ with $0<\gamma\leq T$.
\end{lemma}

\begin{proof}
Since $\lambda_{0}< \tfrac{3}{5}$, by taking $x=(\log \tau)^{2-\varepsilon}$ in Lemma \ref{new1} (where $\varepsilon\!>\!0$ arbitrary) and estimating the sum over primes trivially, we find that
$$ \log^{+}|\zeta(\sigma\!+\!i\tau)| \leq \Big(\frac{1\!+\!\lambda_{0}}{4} + o(1)\Big) \frac{\log \tau}{\log\log \tau} \leq \frac{2}{5} \frac{\log \tau}{\log\log \tau}$$
for $|\tau|$ sufficiently large.  Therefore, we may suppose that $V\leq \frac{2}{5} \frac{\log T}{\log\log T}$, for otherwise the set $\mathcal{S}_{\alpha}(T;V)$ is empty.

We define a parameter
$$
A =A(T,V) = \left\{ \begin{array}{ll}
 \tfrac{1}{2}\log_{3}(T), &\mbox{ if $ \ V \leq \log\log T$,} \\
\tfrac{\log\log T}{2V}\log_{3}(T), &\mbox{ if  $ \ \log\log T < V \leq \tfrac{1}{2}(\log\log T)\log_{3}T$,} \\
1, &\mbox{ if $ \ V > \tfrac{1}{2}(\log\log T)\log_{3}T$,}
       \end{array} \right.
$$
set $x=\min\big(T^{1/2},T^{A/V}\big)$, and put $z=x^{1/\log\log T}$.  Further, we let
$$ S_{1}(s) = \sum_{p\leq z} \frac{1}{p^{s+\frac{\lambda_0}{\log x}}}\frac{\log(x/p)}{\log x} \quad \text{ and } \quad S_{2}(s) = \sum_{z<p\leq x} \frac{1}{p^{s+\frac{\lambda_0}{\log x}}}\frac{\log(x/p)}{\log x}. $$
Then Lemma \ref{new1} implies that
\begin{equation}\label{val1}
 \log^{+}|\zeta(\rho\!+\!\alpha)| \leq |S_{1}(\rho)|+ |S_{2}(\rho)| + \frac{(1\!+\!\lambda_{0})}{2A}V + O\big(\log_{3}T\big)
 \end{equation}
for any non-trivial zero $\rho=\tfrac{1}{2}+i\gamma$ of $\zeta(s)$ with $0<\gamma\leq T$.  Here we have used that $\lambda_{0}\geq 1/2$, $x\leq T^{1/2},$ and $0\leq\Re\alpha-\tfrac{1}{2}\leq (\log T)^{-1}$ which together imply that
$$ \frac{1}{2} \leq \Re (\rho\!+\!\alpha)   \leq \frac{1}{2} +\frac{1}{\log T} \leq \frac{1}{2} + \frac{\lambda_{0}}{\log x}.$$
Since $\lambda_{0}<3/5$, it follows from the inequality in (\ref{val1}) that
 $$ \log^{+}|\zeta(\rho\!+\!\alpha)| \leq |S_{1}(\rho)|+ |S_{2}(\rho)| + \tfrac{4}{5} \tfrac{V}{A} + O\big(\log_{3}T\big).$$
Therefore, if $\rho \in \mathcal{S}_{\alpha}(T;V)$, then either
$$  |S_{1}(\rho)| \geq V\big(1-\tfrac{9}{10A}\big) \quad \quad \text{ or } \quad \quad|S_{2}(\rho)| \geq \tfrac{V}{10 A}.$$
 For simplicity, we put $V_{1}= V\big(1\!-\!\frac{9}{10A}\big)$ and $V_{2}= \frac{V}{10 A}.$
 
Let $N_{1}(T;V)$ be the number of $\rho$ with $0<\gamma\leq T$ such that $|S_{1}(\rho)| \geq V_{1}$ and let $N_{2}(T;V)$ be the number of $\rho$ with $0<\gamma\leq T$ such that $|S_{2}(\rho)| \geq V_{2}$.  We prove the lemma by obtaining upper bounds for the size of the sets $N_{i}(T;V)$ for $i=1$ and $2$ using the inequality 
\begin{equation}\label{val2}
 N_{i}(T;V) \cdot V_{i}^{2k} \leq \sum_{0<\gamma\leq T} |S_{i}(\rho)|^{2k},
 \end{equation}
which holds for any positive integer $k$.  With some restrictions on the size of $k$, we can use Lemma \ref{mvt} to estimate the sums appearing on the right-hand side of this inequality.

We first turn our attention to estimating $N_{1}(T;V).$  If we define the sequence $\alpha_{k}(n) = \alpha_{k}(n,x,z)$ by
$$  \sum_{n\leq z^{k}} \frac{\alpha_{k}(n)}{n^{s}} = \left(\sum_{p\leq z} \frac{1}{p^{s}}\frac{\log x/p}{\log x}\right)^{k},$$
then it is easily seen that $|\alpha_{k}(n)|\leq k!$.  Thus, Lemma \ref{mvt} implies that the estimate
\begin{eqnarray*}
\sum_{0<\gamma\leq T} |S_{1}(\rho)|^{2k} &\ll& N(T) \ k! \ \Big(\sum_{p\leq z}\frac{1}{p}\frac{\log(x/p)}{\log x}\Big)^{k} 
\\
&\ll&  N(T) \ k! \ \Big(\sum_{p\leq z}\frac{1}{p}\Big)^{k} 
\\
&\ll& N(T) \sqrt{k} \Big(\frac{k \log\log T}{e}\Big)^{k}
\end{eqnarray*}
holds for any positive integer $k$ with $z^{k}\leq T(\log T)^{-1}$ and $T$ sufficiently large.  Using (\ref{val2}), we deduce from this estimate that 
\begin{equation}\label{val3}
N_{1}(T;V) \ll N(T)  \sqrt{k} \Big(\frac{k \log\log T}{eV_{1}^{2}}\Big)^{k}.
\end{equation}
It is now convenient to consider separately the case when $V\leq (\log\log T)^{2}$ and the case $V>(\log\log T)^{2}$.  When $V\leq (\log\log T)^{2}$ we choose $k = \lfloor V_{1}^{2}/\log\log T\rfloor$ where, as before, $\lfloor x \rfloor$ denotes the greatest integer less than or equal to $x$.  To see that this choice of $k$ satisfies $z^{k}\leq T(\log T)^{-1},$ we notice from the definition of $A$ that
$$ VA \leq \max \left( V, \ \tfrac{1}{2}(\log\log T)\log_{3}T\right).$$
Therefore, we find that
\begin{eqnarray*}
z^{k} \leq z^{V_{1}^{2}/\log\log T} &=& \exp\left(\tfrac{VA\log T}{(\log\log T)^{2}}\big(1-\tfrac{9}{10 A}\big)^{2}\right)
\\
&\leq & \exp\left(\log T \big(1-\tfrac{9}{10 A}\big)^{2}\right) 
\\
&\leq& T/\log T.
\end{eqnarray*}  
Thus, by (\ref{val3}), we see that for $ V \leq (\log\log T)^{2}$ and $T$ large we have
\begin{equation}\label{val4}
\begin{split}
N_{1}(T;V) \ll N(T)  \frac{V}{\sqrt{\log\log T}} \exp\Big(-\frac{V_{1}^{2}}{\log\log T}\Big).
\end{split}
\end{equation}
When $V\!>\!(\log\log T)^{2}$ we choose $k\!=\! \lfloor 10V\rfloor$.  This choice of $k$ satisfies $z^{k}\leq T(\log T)^{-1}$ since $ z^{10V} =T^{10/\log\log T} \leq T(\log T)^{-1}$
for large $T$.  With this choice of $k$, we conclude from (\ref{val3}) that
\begin{equation}\label{ff}
\begin{split}
N_{1}(T;V) & \ll N(T)  \exp \Big(\tfrac{1}{2} \log V - 10V \log\Big(\tfrac{eV}{1000\log\log T}\Big)\Big) 
\\
& \ll N(T) \exp \big( -10 V\log V + 11 V \log_{3}(T)\big)
\end{split}
\end{equation}
for $T$ sufficiently large.  Since $V > (\log\log T)^{2}$, we have that $\log V \geq 2 \log_{3}(T)$ and thus it follows from (\ref{ff})  that 
\begin{equation}\label{val5}
\begin{split}
N_{1}(T;V) & \ll N(T) \exp\big(-4V\log V\big).
\end{split}
\end{equation}
By combining (\ref{val4}) and (\ref{val5}), we have shown that, for any choice of $V$,
\begin{equation}\label{val6}
\begin{split}
N_{1}(T;V) & \ll N(T)\frac{V}{\sqrt{\log\log T}} \exp\Big(-\frac{V_{1}^{2}}{\log\log T}\Big)+ N(T) \exp\big(-4V\log V\big).
\end{split}
\end{equation}

We now turn our attention to estimating $N_{2}(T;V).$   If we define the sequence $\beta_{k}(n) = \beta_{k}(n,x,z)$ by
$$  \sum_{n\leq x^{k}} \frac{\beta_{k}(n)}{n^{s}} = \left(\sum_{z<p\leq x} \frac{1}{p^{s}}\frac{\log x/p}{\log x}\right)^{k},$$
then it can be seen that $|\beta_{k}(n)|\leq k!$.  Thus, Lemma \ref{mvt} implies that 
\begin{equation}\label{val7}
\begin{split}
\sum_{0<\gamma\leq T} |S_{2}(\rho)|^{2k} &\ll N(T) \ k! \ \Big(\sum_{z<p\leq x}\frac{1}{p}\frac{\log(x/p)}{\log x}\Big)^{k} 
\\
&\ll  N(T) \ k! \ \Big(\sum_{z<p\leq x}\frac{1}{p}\Big)^{k} 
\\
&\ll N(T) \ k! \ \Big(\log_{3}(T) +O(1)\Big)^{k}
\\
&\ll N(T) \ k! \ \big(2\log_{3}(T) \big)^{k}
\\
&\ll N(T) \big( 2k \log_{3}(T)\big)^{k}
\end{split}
\end{equation}
for any natural number $k$ with $x^{k} \leq T/\log T$ and $T$ sufficiently large.  The choice of $k=\lfloor \frac{V}{A}\!-\!1\rfloor$ satisfies $x^{k} \leq T/\log T$ when $T$ is large.  To see why, recall that $A\geq 1, x=T^{A/V},$ and $V\leq\frac{2}{5}\frac{\log T}{\log\log T}$.  Therefore, 
$$ x^{k} \leq x^{(V/A-1)} \leq T^{1-A/V} \leq T^{1-1/V} = T(\log T)^{-5/2} \leq T(\log T)^{-1}.$$
Also, observing that $A\leq\tfrac{1}{2}\log_{3}(T)$ and recalling that $V\geq\sqrt{\log\log T}$, with this choice of $k$ and $T$ large, it follows from (\ref{val2}) that
\begin{equation}\label{val8}
\begin{split}
N_{2}(T;V) &\ll N(T) \Big(\frac{10A}{V}\Big)^{2k} \big( 2k \log_{3}(T)\big)^{k}
\\
&\ll N(T) \exp\Big( -2k \log(\tfrac{V}{10A}) + k \log(2k\log_{3}(T))\Big)
\\
&\ll N(T) \exp\Big( -2\tfrac{V}{A} \log(\tfrac{V}{10A}) +2\log\tfrac{V}{10A}+\tfrac{V}{A} \log\big(\tfrac{2V}{A}\log_{3}(T)\big)\Big)
\\
&\ll N(T) \exp\big(-\tfrac{V}{2A} \log V\big).
\end{split}
\end{equation}

Using our estimates for $N_{1}(T;V)$ and $N_{2}(T;V)$ we can now complete the proof of the lemma by checking the various ranges of $V$.  By combining (\ref{val6}) and (\ref{val8}), we see that
\begin{equation}\label{val9}
\begin{split}
\#\mathcal{S}_{\alpha}(T;V) &\ll N(T)\frac{V}{\sqrt{\log\log T}} \exp\big(-\tfrac{V_{1}^{2}}{\log\log T}\big) + N(T) \exp\big(-4V\log V\big)
\\
& \quad \quad\quad \quad + N(T) \exp\big(-\tfrac{V}{2A} \log V\big).
\end{split}
\end{equation}
If $\sqrt{\log\log T}\leq V \leq \log\log T$, then $A=\tfrac{1}{2}\log_{3}(T)$ and (\ref{val9}) implies that, for $T$ sufficiently large,
\begin{equation}\label{val10}
\begin{split}
 \#\mathcal{S}_{\alpha}(T;V) &\ll N(T)\frac{V}{\sqrt{\log\log T}} \exp\left(-\tfrac{V^{2}}{\log\log T}\Big(1-\tfrac{9}{5\log_{3}T}\Big)^{2}\right)
 \\
 &\ll N(T)\frac{V}{\sqrt{\log\log T}} \exp\left(-\tfrac{V^{2}}{\log\log T}\Big(1-\tfrac{4}{\log_{3}T}\Big)\right).
 \end{split}
 \end{equation}
If $\log\log T < V\leq \tfrac{1}{2}(\log\log T)\log_{3}(T)$, then $A = \tfrac{\log\log T}{2V}\log_{3}(T)$ and we deduce from (\ref{val9}) that
\begin{equation}\label{val11}
\begin{split}
\#\mathcal{S}_{\alpha}(T;V) &\ll N(T)\frac{V}{\sqrt{\log\log T}} \exp\left(-\tfrac{V^{2}}{\log\log T}\Big(1-\tfrac{9}{5(\log\log T)\log_{3}T}\Big)^{2}\right)
\\
& \quad \quad + N(T) \exp\left(-\tfrac{V^{2} \log V}{(\log\log T)\log_{3}T}\right) + N(T) \exp\big(-4V\log V\big).
 \end{split}
 \end{equation}
 For $V$ in this range, $ \frac{ \log V}{(\log\log T)\log_{3}T} > \frac{1}{\log\log T} \text{ and } \frac{V}{\log V} < \log\log T$, so (\ref{val11}) implies that
 \begin{equation}\label{val12}
\begin{split}
\#\mathcal{S}_{\alpha}(T;V) &\ll N(T)\frac{V}{\sqrt{\log\log T}} \exp\left(-\tfrac{V^{2}}{\log\log T}\Big(1-\tfrac{9}{5(\log\log T)\log_{3}T}\Big)^{2}\right)
\\
& \ll  N(T)\frac{V}{\sqrt{\log\log T}} \exp\left(-\tfrac{V^{2}}{\log\log T}\Big(1-\tfrac{4}{(\log\log T)\log_{3}T}\Big)\right).
 \end{split}
 \end{equation}
Finally, if $V \geq \tfrac{1}{2}(\log\log T)\log_{3}T$, then $A=1$ and we deduce from (\ref{val9}) that
 \begin{equation}\label{val13}
\begin{split}
\#\mathcal{S}_{\alpha}(T;V) &\ll N(T) \exp\left(\log V-\tfrac{V^{2}}{100 \log\log T}\right) + N(T) \exp\left(-\tfrac{V}{2} \log V\right). \end{split}
 \end{equation}
Certainly,  if $V \geq \tfrac{1}{2}(\log\log T)\log_{3}T$ then we have that $ \frac{V^{2}}{100 \log\log T} - \log V > \frac{1}{201} V\log V $ for $T$ sufficiently large and so it follows from (\ref{val13}) that
\begin{equation}\label{val14}
\begin{split}
\#\mathcal{S}_{\alpha}(T;V) &\ll N(T) \exp\left(-\tfrac{V}{201} \log V\right) .
\end{split}
\end{equation}
The lemma now follows from the estimates in (\ref{val10}), (\ref{val12}), and (\ref{val14}).
\end{proof}


\section{The proof of Theorem \ref{th2}}

Using Lemma \ref{vd}, we first prove Theorem \ref{th2} in the case where $|\alpha|\leq 1$ and $0\leq\Re\alpha \leq (\log T)^{-1}$.  Then, from this result, the case when  $-(\log T)^{-1}\leq \Re \alpha <0$ can be deduced from the functional equation for $\zeta(s)$ and Stirling's formula for the gamma function.  In what follows, $k\in\mathbb{R}$ is fixed and we let $\varepsilon>0$ be an arbitrarily small positive constant which may not be the same at each occurrence.  

First, we partition the real axis into the intervals $I_{1} = (-\infty, 3],  I_{2}=(3, 4k \log\log T],$ and $I_{3}= (4k\log\log T, \infty)$ and set
$$ \Sigma_{i} = \sum_{\nu\in I_{i} \cap \mathbb{Z}} e^{2k\nu} \cdot \#\mathcal{S}_{\alpha}(T,\nu)$$
for $i=1,2,$ and $3$.  Then we observe that
\begin{equation} \label{obsv}
\sum_{0<\gamma\leq T} \big|\zeta(\rho\!+\!\alpha)\big|^{2k} \leq \sum_{\nu\in\mathbb{Z}} e^{2k\nu} \Big[  \#\mathcal{S}_{\alpha}(T,\nu) - \#\mathcal{S}_{\alpha}(T,\nu\!-\!1)\Big] \leq \Sigma_{1}+\Sigma_{2}+\Sigma_{3}.
\end{equation}  Using the trivial bound $\#\mathcal{S}_{\alpha}(T,\nu) \leq N(T)$, which holds for every $\nu\in\mathbb{Z}$, we find that $\Sigma_{1} \leq e^{6k} N(T)$.  To estimate $\Sigma_{2}$, we use the bound $$\#\mathcal{S}_{\alpha}(T,\nu) \ll N(T) (\log T)^{\varepsilon} \exp\Big(\frac{-\nu^{2}}{\log\log T}\Big)$$ which follows from the first two cases of Lemma \ref{vd} when $\nu\in I_{2}\cap \mathbb{Z}$.  From this, it follows that
\begin{equation*}
\begin{split}
\Sigma_{2} &\ll N(T) (\log T)^{\varepsilon} \int_{3}^{4k\log\log T} \exp\left(2ku-\tfrac{u^{2}}{\log\log T}\right) du
\\
&\ll N(T) (\log T)^{\varepsilon} \int_{0}^{4k} (\log T)^{u(2k-u)}  \ du
\\
&\ll N(T) (\log T)^{k^{2}+\varepsilon}
\end{split}
\end{equation*}
When $\nu\in I_{3}\cap \mathbb{Z}$, the second two cases of Lemma \ref{vd} imply that $$\#\mathcal{S}_{\alpha}(T,\nu) \ll N(T) (\log T)^{\varepsilon} e^{-4k\nu}.$$ Thus, 
\begin{equation*}
\begin{split}
\Sigma_{3} &\ll N(T) (\log T)^{\varepsilon}\int_{4k\log\log T}^{\infty } e^{-2ku} \  du \ll N(T) (\log T)^{-8 k^{2} + \varepsilon}.
\end{split}
\end{equation*}
In light of (\ref{obsv}), by collecting estimates, we see that
\begin{equation}\label{obsv2}
\sum_{0<\gamma\leq T} \big|\zeta(\rho\!+\!\alpha)\big|^{2k} \ll N(T) (\log T)^{k^{2}+\varepsilon}
\end{equation}
for every $k>0$ when $|\alpha|\leq 1$ and $0\leq \Re\alpha \leq (\log T)^{-1}$.  

The functional equation for the zeta-function states that $ \zeta(s) = \chi(s)\zeta(1-s)$ where $\chi(s) = 2^{s}\pi^{s-1}\Gamma(1\!-\!s)\sin\big(\tfrac{\pi s}{2}\big)$.  Stirling's asymptotic formula for the gamma function (see Appendix A.7 of Ivi\`{c} \cite{Ivic85}) can be used to show that
\begin{equation*}
\big| \chi(\sigma\!+\!it) \big| = \Big(\frac{|t|}{2\pi}\Big)^{1/2-\sigma}\Big(1+O\Big(\frac{1}{|t|}\Big)\Big)
\end{equation*}
uniformly for $-1\leq \sigma\leq 2$ and $|t|\geq 1$. Using the Riemann Hypothesis, we see that
\begin{equation*}
\begin{split}
 \big|\zeta(\rho\!+\!\alpha)\big| &= \big|\chi(\rho\!+\!\alpha) \zeta(1\!-\!\rho\!-\!\alpha) \big| 
 \\
 &=  \big|\chi(\rho\!+\!\alpha) \zeta(\bar{\rho}\!-\!\alpha) \big| 
 \\
 & =  \big|\chi(\rho\!+\!\alpha) \zeta(\rho\!-\!\bar{\alpha}) \big| 
 \\
 &\leq C \big|\zeta(\rho\!-\!\bar{\alpha}) \big|
\end{split}
\end{equation*}
for some absolute constant $C>0$ when $|\alpha|\leq1$, $|\Re\alpha -\tfrac{1}{2}|\leq (\log T)^{-1}$, and $0<\gamma\leq T$.  Consequently, for $-(\log T)^{-1}\leq\Re \alpha<0$, 
\begin{equation}\label{obsv3}
 \sum_{0<\gamma\leq T} \big|\zeta(\rho\!+\!\alpha)\big|^{2k} \leq C^{2k}\cdot \sum_{0<\gamma\leq T} \big|\zeta(\rho\!-\!\bar{\alpha})\big|^{2k}.
 \end{equation}
Applying the inequality in (\ref{obsv2}) to the right-hand side of (\ref{obsv3}) we see that 
\begin{equation}\label{obsv4}
\sum_{0<\gamma\leq T} \big|\zeta(\rho\!+\!\alpha)\big|^{2k} \ll_{k} N(T) (\log T)^{k^{2}+\varepsilon}
\end{equation}
for every $k>0$ when $|\alpha|\leq 1$ and $-(\log T)^{-1} \leq \Re\alpha< 0$.  The theorem now follows from the estimates in (\ref{obsv2}) and (\ref{obsv4}).


\section{Theorem \ref{th2} implies Theorem \ref{th1}}

Theorem \ref{th1} can now be established as a simple consequence of Theorem \ref{th2} and the following lemma.

\begin{lemma}\label{CE}
Assume the Riemann Hypothesis.  Let $k, \ell \in \mathbb{N}$ and let $R>0$ be arbitrary.  Then we have
\begin{equation}\label{CE1}
 \sum_{0<\gamma\leq T} \big|\zeta^{(\ell)}(\rho)\big|^{2k} \leq \Big(\frac{\ell !}{R^{\ell}}\Big)^{2k} \cdot \left[\max_{|\alpha|\leq R} \ \sum_{0<\gamma\leq T} \big|\zeta(\rho\!+\!\alpha)\big|^{2k}\right].
 \end{equation}
\end{lemma}

\begin{proof}
Since the function $\zeta^{(\ell)}(s)$ is real when $s\in\mathbb{R}$, $\zeta^{(\ell)}(\bar{s})=\overline{\zeta^{(\ell)}(s)}$.  Hence, assuming the Riemann Hypothesis, the identity 
\begin{equation}\label{CE3}
\big|\zeta^{(\ell)}(1\!-\!\rho\!+\!\alpha)\big|=\big|\zeta^{(\ell)}(\bar{\rho}\!+\!\alpha)\big|=  \big|\zeta^{(\ell)}(\rho\!+\!\overline{\alpha})\big|
\end{equation}
holds for any non-trivial zero $\rho$ of $\zeta(s)$ and any $\alpha\in\mathbb{C}.$  For each positive integer $k$, let $\vec{\alpha}_{k} = (\alpha_{1},\alpha_{2},\ldots,\alpha_{2k})$ and define
$$ \mathcal{Z}\big(s;\vec{\alpha}_{k}\big) = \prod_{i=1}^{k} \zeta(s\!+\!\alpha_{i})\zeta(1\!-\!s\!+\!\alpha_{i+k}).$$
If we suppose that each $|\alpha_{i}|\leq R$ for $i=1,\ldots, 2k$ and apply H\"{o}lder's inequality in the form
\begin{equation*}
\left|\sum_{n=1}^{N} \Bigg(\prod_{i=1}^{2k} f_{i}(s_{n})\Bigg)\right| \leq \prod_{i=1}^{2k} \Big(\sum_{n=1}^{N} |f_{i}(s_{n})|^{2k}\Big)^{\frac{1}{2k}},
\end{equation*}
we see that (\ref{CE3}) implies that
\begin{equation}\label{CE4}
\begin{split}
\left|\sum_{0<\gamma\leq T}  \mathcal{Z}\big(\rho;\vec{\alpha}_{k}\big)\right| & \leq \prod_{i=1}^{k} \left(\sum_{0<\gamma\leq T} \big|\zeta(\rho\!+\!\alpha_{i})\big|^{2k}\right)^{\frac{1}{2k}}\left(\sum_{0<\gamma\leq T} \big|\zeta(\rho\!+\!\overline{\alpha_{k+i}})\big|^{2k}\right)^{\frac{1}{2k}}
\\
& \leq \max_{|\alpha|\leq R}  \sum_{0<\gamma\leq T} \big|\zeta(\rho\!+\!\alpha)\big|^{2k}
\end{split}
\end{equation}
In order to prove the lemma, we first rewrite the left-hand side of equation (\ref{CE1}) using the function $ \mathcal{Z}\big(s;\vec{\alpha}_{k}\big)$ and then apply the inequality in (\ref{CE4}).  By Cauchy's integral formula and another application of (\ref{CE3}), we see that
\begin{equation}\label{CE2}
\begin{split}
 \sum_{0<\gamma\leq T} \big|\zeta^{(\ell)}(\rho)\big|^{2k} &= \sum_{0<\gamma\leq T}\bigg(\prod_{i=1}^{k} \zeta^{(\ell)}(\rho)\zeta^{(\ell)}(1\!-\!\rho)\bigg)
\\
&= \frac{(\ell !)^{2k}}{(2\pi i)^{2k}}\mathop{\int_{\mathscr{C}_{1}}\!\!\!\cdots\!\int_{\mathscr{C}_{2k}}}\!\!\left(\sum_{0<\gamma\leq T} \mathcal{Z}\big(\rho;\vec{\alpha}_{k}\big) \right)\prod_{i=1}^{2k} \frac{d\alpha_{i}}{\alpha_{i}^{\ell + 1}}
\end{split}
\end{equation}
where, for each $i=1,\ldots,2k$, the contour $\mathscr{C}_{i}$ denotes the positively oriented circle in the complex plane centered at $0$ with radius $R$.  Now, combining (\ref{CE4}) and (\ref{CE2}) we find that
\begin{equation*}
\begin{split}
 \sum_{0<\gamma\leq T} \big|\zeta^{(\ell)}(\rho)\big|^{2k} &\leq \Big(\frac{\ell !}{2\pi}\Big)^{2k}\cdot \left[\max_{|\alpha|\leq R} \ \sum_{0<\gamma\leq T} \big|\zeta(\rho\!+\!\alpha)\big|^{2k}\right] \cdot \mathop{\int_{\mathscr{C}_{1}}\!\!\!\cdots\!\int_{\mathscr{C}_{2k}}} \prod_{i=1}^{2k} \frac{d\alpha_{i}}{|\alpha_{i}|^{\ell + 1}}  
 \\
 &\leq \Big(\frac{\ell !}{2\pi}\Big)^{2k}\cdot  \left[\max_{|\alpha|\leq R} \ \sum_{0<\gamma\leq T} \big|\zeta(\rho\!+\!\alpha)\big|^{2k}\right] \cdot\Big(\frac{2\pi}{R^{\ell}}\Big)^{2k}
 \\
 &\leq  \Big(\frac{\ell !}{R^{\ell}}\Big)^{2k} \cdot \left[\max_{|\alpha|\leq R} \ \sum_{0<\gamma\leq T} \big|\zeta(\rho\!+\!\alpha)\big|^{2k}\right],
\end{split}
\end{equation*}
as claimed.
\end{proof}

\noindent\textit{Proof of Theorem \ref{th1}.}  Let $k\in\mathbb{N}$ and set $R= (\log T)^{-1}$.  Then, it follows from Theorem \ref{th2} and Lemma \ref{CE} that
\begin{equation}
\frac{1}{N(T)}\sum_{0<\gamma\leq T} \big|\zeta^{(\ell)}(\rho)\big|^{2k} \ll_{k,\ell,\varepsilon} (\log T)^{k(k+2\ell)+ \varepsilon}
\end{equation}
for any $\ell\in\mathbb{N}$ and for $\varepsilon>0$ arbitrary.  Theorem \ref{th1} now follows by setting $\ell=1$.
\end{onehalfspacing}


\begin{thebibliography}{100}

\bibitem{CS} J. B. Conrey and N. C. Snaith, Applications of the L-functions ratios conjectures. Proc. London Math. Soc. {\bf 94} (2007), no. 3, 594--646. 

\bibitem{Fuj95} A. Fujii, On a mean value theorem in the theory of the Riemann zeta function. Comment. Math. Univ. St. Paul. {\bf 44} (1995), no. 1, 59--67. 

\bibitem{Gar03} M. Z. Garaev, One inequality involving simple zeros of $\zeta(s)$.  Hardy-Ramanujan J. {\bf 26} (2003-2004), 18--22. 

\bibitem{GS05} 
R. Garunk\u{s}tis and J. Steuding, Simple zeros and discrete moments of the derivative of the Riemann zeta-function. J. of Number Theory {\bf 115} (2005), 310--321.
 
\bibitem{Gon84}
 S. M. Gonek, Mean values of the Riemann zeta-function and its derivatives. Invent. Math. {\bf 75} (1984), 123--141. 

\bibitem{Gon85} S. M. Gonek, A formula of Landau and mean values of $\zeta(s)$. Topics in Analytic Number Theory (Austin, Tex.) (S. W. Graham and J. D. Vaaler, eds.), Univ. Texas Press, 1985, 92--97. 

\bibitem{Gon89} S. M. Gonek, On negative moments of the Riemann zeta-function. Mathematika {\bf 36} (1989), 71--88.
 
\bibitem{Gon93} S. M. Gonek, An explicit formula of Landau and its applications to the theory of the zeta function. Contemp. Math {\bf 143} (1993), 395--413.
 
\bibitem{Gon99} S. M. Gonek, The second moment of the reciprocal of the Riemann zeta-function and its derivative. Talk at Mathematical Sciences Research Institute, Berkeley, June 1999. 

\bibitem{Hej89} 
D. Hejhal, On the distribution of $\log |\zeta'(\tfrac{1}{2}+it)|$. Number Theory, Trace Formulas, and Discrete Groups - Proceedings of the 1987 Selberg Symposium (K. E. Aubert, E. Bombieri, and D. M. Goldfeld, eds.), Academic Press, 1989, 343--370.

\bibitem{Hug03} 
C. P. Hughes, Random matrix theory and discrete moments of the Riemann zeta function. Journal of Physics A: Math. Gen. {\bf 36} (2003), 2907--2917. 

\bibitem{HKO00}
C. P. Hughes, J. P. Keating, and N. OÕConnell, Random matrix theory and the derivative of the Riemann zeta-function. Proc. Roy. Soc. London A {\bf 456} (2000), 2611--2627.
 
\bibitem{Ivic85}
A. Ivi\`{c}, The Riemann zeta-function: Theory and applications. Dover Publications,1985.
 
\bibitem{LS04} A. Laurin\v{c}ikas and J. Steuding, A note on the moments of $\zeta'(1/2+i\gamma)$. Publ. Inst. Math. Beograd {\bf 76} (2004), no. 90, 57--63.
 
\bibitem{Mil07a} M. B. Milinovich, Moments of the Riemann zeta-function and its derivatives. Preprint. 

\bibitem{MN07}
M. B. Milinovich and N. Ng, Lower bounds for the moments of $\zeta'(\rho)$. Submitted, 2007. Preprint available on the {\tt arXiv} at {\tt http://arxiv.org/abs/0706.2321}.

\bibitem{Ng04}
N. Ng, The fourth moment of $\zeta'(\rho)$.  Duke Math J. {\bf 125} (2004), 243--266.

\bibitem{Sel44} 
A. Selberg, On the remainder in the formula for $N(T)$, the number of zeros of $\zeta(s)$ in the strip $0 < t < T$.  Avhandlinger Norske Vid. Akad. Olso. {\bf 1} (1944), 1--27. 
\bibitem{Sel46}
A. Selberg, Contributions to the theory of the Riemann zeta-function. Archiv. Math. Naturvid. {\bf 48} (1946), 89--155. 
\bibitem{Sel91}
A. Selberg, Collected Papers (vol. II), ch. ``Old and new conjectures and results about a class of Dirichlet series.'' pp. 47--63, Springer Verlag, 1991.
 
\bibitem{SS04}
 R. \v{S}le\v{z}evi\v{c}iene and J. Steuding, Short series over simple zeros of the Riemann zeta-function.  Indagationes Math. {\bf 15} (2004), 129--132. 
\bibitem{Sou07} K. Soundararajan, Moments of the Riemann zeta-function. To appear in Ann. of Math., 2008. Preprint available on the {\tt arXiv} at {\tt http://arxiv.org/abs/math/0612106}.
\bibitem{Tit86} E. C. Titchmarsh, The Theory of the Riemann Zeta-Function, 2nd Ed., revised by D. R. Heath-Brown, Clarendon, Oxford 1986. 
\end{thebibliography}

\end{document}